\documentclass[11pt,reqno]{amsart}
\usepackage{amssymb,latexsym}
\usepackage{graphicx}
\usepackage{tikz}
\usepackage{adjustbox}
\usepackage{forest}
\usetikzlibrary{trees}
\usepackage{tikz-qtree}
\usepackage[colorlinks=true,
linkcolor=webgreen,
filecolor=webbrown,
citecolor=webgreen]{hyperref}

\definecolor{webgreen}{rgb}{0,.5,0}
\definecolor{webbrown}{rgb}{.6,0,0}

\usepackage{color}
\usepackage{float}

\usepackage{url}		
\usepackage{amsmath}
\makeatletter
\def\legendre@dash#1#2{\hb@xt@#1{%
  \kern-#2\p@
  \cleaders\hbox{\kern.5\p@
    \vrule\@height.2\p@\@depth.2\p@\@width\p@
    \kern.5\p@}\hfil
  \kern-#2\p@
  }}
\def\@legendre#1#2#3#4#5{\mathopen{}\left(
  \sbox\z@{$\genfrac{}{}{0pt}{#1}{#3#4}{#3#5}$}%
  \dimen@=\wd\z@
  \kern-\p@\vcenter{\box0}\kern-\dimen@\vcenter{\legendre@dash\dimen@{#2}}\kern-\p@
  \right)\mathclose{}}
\newcommand\legendre[2]{\mathchoice
  {\@legendre{0}{1}{}{#1}{#2}}
  {\@legendre{1}{.5}{\vphantom{1}}{#1}{#2}}
  {\@legendre{2}{0}{\vphantom{1}}{#1}{#2}}
  {\@legendre{3}{0}{\vphantom{1}}{#1}{#2}}
}
\def\dlegendre{\@legendre{0}{1}{}}
\def\tlegendre{\@legendre{1}{0.5}{\vphantom{1}}}
\makeatother

\newcommand{\Ccal}{{\mathcal C}}

\newcommand{\Z}{{\mathbb Z}}

 \newtheorem{theorem}{Theorem}[section]
 \newtheorem{corollary}[theorem]{Corollary}
 \newtheorem{lemma}[theorem]{Lemma}
 
 \theoremstyle{definition}
 
 \theoremstyle{remark}

 \numberwithin{equation}{section}

\tikzstyle{level 1}=[level distance=1.5cm, sibling distance=1.5cm]
\tikzstyle{level 2}=[level distance=1.5cm, sibling distance=2cm]

\tikzstyle{bag} = [text width=4em, text centered]
\tikzstyle{end} = [circle, minimum width=3pt,fill, inner sep=0pt]

\begin{document}

\title{The Markoff equation over polynomial rings}
\author{Ricardo Concei\c c\~ao}
\address{Department of Mathematics\\
                Gettysburg College\\
                Gettysburg, Pa\\
                17325, USA}
\email{rconceic@gettysburg.edu}
\author{Rachael Kelly}
\address{Department of Mathematics\\
                Gettysburg College\\
                Gettysburg, Pa\\
                17325, USA}
\email{kellra01@gettysburg.edu}
\author{Samuel VanFossen}
\address{Department of Mathematics\\
                Gettysburg College\\
                Gettysburg, Pa\\
                17325, USA}
\email{vanfsa01@gettysburg.edu}



\date{}
\begin{abstract}
When $A=3$, the positive integral solutions of the so-called Markoff equation 
$$M_A:x^2 + y^2 + z^2 = Axyz$$
 can be generated from the single solution $(1,1,1)$ by the action of certain automorphisms of the hypersurface. Since Markoff's proof of this fact, several authors have showed that the structure of $M_A(R)$, when $R$ is $\Z[i]$ or certain orders in number fields, behave in a similar fashion. Moreover, for $R=\Z$ and $R=\Z[i]$, Zagier and Silverman, respectively, have found asymptotic formulae for the number of integral points of bounded height. In this paper, we investigate these problems when $R$ is a polynomial ring over a field $K$ of odd characteristic. We characterize the set $M_A(K[t])$ in a similar fashion as Markoff and previous authors. We also give an asymptotic formula that is similar to Zagier's and Silverman's formula.
\end{abstract}

\maketitle
\section{Introduction}\label{Introduction}


The Markoff equation 
\begin{equation}\label{eq:Markoffpol}
M_A:x^2 + y^2 + z^2 = Axyz, 
\end{equation}
  with $A\neq 0$, has been a subject of close scrutiny in mathematics for over a century. Interest in this equation grew out of Markoff's work relating the set of integral solutions  $M_3(\Z)$ to questions in diophantine approximation \cite{aigner,cassels}.  Markoff also proved that all  non-zero integral solutions of \eqref{eq:Markoffpol}, if they exist, can be generated from a single fundamental solution by permutation of the coordinates,  a double change of sign
  $$
  (x,y,z)\longmapsto (-x,-y,z),
  $$
   and the branching automorphism 
\begin{equation} \label{eq:pred}
  \rho:(x,y,z)\longmapsto (x,y,Axy-z).
\end{equation}
More precisely, if we let $G_A$ be the group generated by the above automorphisms, then Markoff's result says that 
$$
M_A(\Z)\backslash \{0\}= \left\{
\begin{array}{ll}
 G_A\{(1,1,\pm1)\}, & \text{if $A=\pm 3$};\\
 G_A\{(3,3,\pm3)\}, & \text{if $A=\pm 1$};\\
 \emptyset, & \text{otherwise.}
\end{array}\right.
$$
where $G_A(P)$ denotes  the  $G_A$-orbit of the point $P$.

%
%
%
 
Given this non-trivial description of the  solutions of the Markoff equation over $\Z$, a natural question is to find other rings for which  the solutions of \eqref{eq:Markoffpol} can be characterized in an analogous  way. This has been done for orders in quadratic imaginary fields \cite{silverman} and, more generally, orders in number fields \cite{baragar}. Additionally, several authors have studied the solutions of \eqref{eq:Markoffpol} and its generalizations over finite rings \cite{carlitz, baoulina}. 

In this paper, we give a similar characterization of the solutions of \eqref{eq:Markoffpol} over the polynomial ring $K[t]$, with  $K$ a field of characteristic $\neq 2$.  Because we are interested on ``integral'' solutions of the Markoff equation over a polynomial ring, it is natural that we focus on the set  $M_A(K[t])$  of \emph{non-constant polynomial solutions}  of \eqref{eq:Markoffpol}; that is, we investigate the triples  $(x,y,z)$ such that $x,y,z\in K[t]$ are not all constants. To state one of our main result, we  define 
$$
G_A(S)=\{g(s): g\in G_A, s \in S\},
$$ 
for  $S$  a subset of $M_A(K[t])$ and $G_A$ the group defined above.

\begin{theorem}\label{thm:pol_tree} Let $A$ be a non-zero element of $K[t]$.
\begin{enumerate}
\item The Markoff polynomial equation \eqref{eq:Markoffpol} has a non-constant solution over $K[t]$ if and only if $i=\sqrt{-1}\in K$. 
 \item If $A$ is non-constant then 
 $$
 M_A(K[t])=G_A\{(f,if,0): f\in K[t]\backslash K\}.
 $$

  \item If $A$ is a non-zero constant then $ M_A(K[t])$ is equal to
   $$
G_A\left\{\left(f,af \pm \frac{2ai}{A},\frac{2a}{A}\right): f\in K[t]\backslash K, a=\pm 1\right\}\cup G_A\{(f,\pm if,0): f\in K[t]\backslash K\}
 $$
\end{enumerate}
\end{theorem}

Theorem \ref{thm:pol_tree} is proved in the next section.

A natural question that arises from the infinitude of $M_A(R)$, for a general ring $R$, is the estimation of the number of solutions  of bounded height. In this direction, we have the following two results
\begin{theorem}[Zagier \cite{zagier}] For some constant $C$,
 $$
 \#\{(x,y,z)\in M_3(\Z):|x|,|y|,|z|\leq H\}=C(\log H)^2+O(\log H(\log\log H)^2)
 $$
 as $H\longrightarrow\infty$.
\end{theorem}

\begin{theorem}[Silverman \cite{silverman}]
Let $a\in \Z[i]$ with $|a|\geq 4$. Then
  $$
 \#\{(x,y,z)\in M_a(\Z[i]):|x|,|y|,|z|\leq H,\gcd(x,y,z)=1\}\gg\ll(\log H)^2
 $$
  as $H\longrightarrow\infty$.
\end{theorem}

Over $K[t]$, we have a few different choices for a counting function for the number of solutions of bounded height $H$.  For general $K$,  we count the number of triples $(\deg x,\deg y,\deg z)$ where $(x,y,z)\in M_A(K[t])$ and $\max\{\deg x,\deg y,\deg z\}\leq H.$ When $K$ is finite,  it is more natural to count the number of triples $(x,y,z)\in M_A(K[t])$ with $\max\{\deg x,\deg y,\deg z\}= H$. In Section \ref{sec:countingsol}, we develop a machinery that can be used to solve many of the counting problems associated to these functions. We highlight the following result because of its similarity with Zagier's and Silverman's Theorems.

\begin{theorem}[Corollary \ref{cor:counting}]\label{thm:counting} For $P=(x,y,z)\in M_A(K[t])$, we let $S(P)=(\deg x,\deg y,\deg z)$ and $h(P)=\max(\deg x,\deg y,\deg z)$.
  If $A$ is a non-zero constant then
 $$
\#\{S(P):P\in M_A(k[t]),h(P)\leq H\} \sim \frac{1}{4}H^2,
 $$
 as $H\longrightarrow \infty$.
\end{theorem}

\section{Proof of Theorem \ref{thm:pol_tree}}\label{sec:Markoff_tree}


Recall that $K$ is a field of characteristic $\neq 2$ and $K[t]$ is the ring of polynomials in the indeterminate $t$ over $K$. Let $P=(x,y,z)$ be a solution to \eqref{eq:Markoffpol}. The \emph{height of $P$} is the integer $h(P)=\max\{\deg x,\deg y,\deg z\}$. We say that $P$ is a \emph{Markoff triple} if $h(P)>0$ and $\deg x\leq \deg y\leq \deg z$.  Notice that all solutions of \eqref{eq:Markoffpol} with positive height become a Markoff triple after permutation of coordinates. If $P$ is a Markoff triple then, after a  permutation of coordinates, the triple $\rho(P)=(x,y,Axy-z)$ given by \eqref{eq:pred}  is a Markoff triple  called \emph{the predecessor of $P$}. 
If $\deg y=\deg z$ then $P$ is called  a \emph{fundamental (Markoff) triple}.

The next result shows that 
$$
M_A(K[t])=G_A(\mathcal{F}),
$$
where $\mathcal{F}$ is the set of all fundamental triples.
\begin{lemma}\label{lem:orbit}
Let $P$ be a Markoff triple. Then there exists $g\in G_A$ and a fundamental Markoff triple $R$ such that $P=g(R)$.
\end{lemma}
\begin{proof}
We first show that if $P$ is a non-fundamental Markoff triple  then there exists a Markoff triple $\tilde{P}$ such that $h(P)>h(\tilde{P})$.

If $P=(x,y,z)$ is not a fundamental triple then $x\neq 0$, from \eqref{eq:Markoffpol}, and $0\leq \deg x\leq \deg y<\deg z$. By equating degrees in \eqref{eq:Markoffpol}, we have
 $$
 \deg z=\deg A+\deg x+\deg y.
 $$
Also, if $a, a_x,a_y$ and $a_z$ are the respective leading coefficients of $A,x,y$ and $z$, then \eqref{eq:Markoffpol} implies 
 $$
 a_z=aa_xa_y.
 $$
 This shows that  ${P}_1=\rho(P)$ satisfies $h(P)>h({P}_1)$. Let $\pi$ be a permutation of coordinates such that $\tilde{P}=\pi({P}_1)$ is a Markoff triple. Then 
 $h(P)>h(\tilde{P})=h({P}_1)$, as desired.

To finish the proof of the lemma, given a Markoff triple $P$,  we construct a sequence of Markoff triples $P_i$, $i\geq 0$, as follows. We let $P_0=P$ and $P_{i+1}=\pi_i\rho (P_{i})$, where $\pi_i$ is the permutation that changes $\rho(P_i)$ into a Markoff triple. If all $P_i$'s were not fundamental triples then, by the above argument, we would arrive at infinite sequence of decreasing non-negative integers
$$
h(P)>h(P_1)>h(P_2)>h(P_3)>\cdots
$$
contradicting the well-ordering principle. Therefore, there is a non-negative integer $n$ such that $P_n$ is a fundamental Markoff triple. If we let $g=g_0g_1\cdots g_{n-1}$, where $g_i=\rho^{-1}\pi_i^{-1}$ then
$$
P=g(P_n)
$$
and the result follows.
\end{proof}

Therefore, to characterize $M_A(K[t])$ we need to characterize all fundamental triples, which we do in the following sequence of results.

\begin{lemma}\label{lem:xzero}
Let $(x,y,z)$ be a fundamental Markoff triple.
\begin{enumerate}
 \item   If $x\neq 0$ then $A$ and $x$ are constants.
 \item If $x=0$ then $i\in K$ and there exists $f\in K[t]\backslash K$ such that $y=\pm if$ and $z=f$.
\end{enumerate}

\end{lemma}
\begin{proof}
Any Markoff triple $(x,y,z)$ with $x\neq 0$ satisfies $y,z\neq 0$ and, by equating degrees in  \eqref{eq:Markoffpol},
$$
\deg z\geq \deg A+\deg x+\deg y.
$$
 If additionally $(x,y,z)$ is a fundamental triple then $0\geq \deg A+\deg x$, which implies that $A$ and $x$ are constants. This proves part (1).

For part (2), notice that if $x=0$ then  \eqref{eq:Markoffpol} implies  $y^2=-z^2$ and $y,z\in K[t]\backslash K$. By comparing the leading coefficients of $y$ and $z$, we see that $i\in K$. To finish the proof, we notice that the solutions of $y^2=-z^2$ satisfies $y=\pm if$ and $z= f$, for some $f\in K[t]$.
\end{proof}

\begin{lemma}\label{lem:funsol1}
 Suppose $A$ is non-constant. Then \eqref{eq:Markoffpol} has a fundamental Markoff triple  if and only if $i\in K$. Moreover, if $i\in K$ then a fundamental Markoff triple is of the form
 $$
 (0,\pm if,f)
 $$
 for some $f \in K[t]\backslash K$.
\end{lemma}
\begin{proof}
If $i\in K$ then for any $f\in K[t]\backslash K$, the Markoff triple $(0,if,f)$ is a fundamental triple.

Conversely, suppose $(x,y,z)$ is a fundamental triple. Because of our assumption on $A$,  Lemma \ref{lem:xzero} implies that $x=0$. Therefore $i\in K$ and $(x,y,z)$ has the desired form.
 \end{proof}

Notice that the previous result and Lemma \ref{lem:orbit} proves part (2) and part (1) of Theorem \ref{thm:pol_tree}, for $A$ non-constant.
 
 \begin{lemma}\label{lem:funsol2}
 Suppose $A$ is constant. Then \eqref{eq:Markoffpol} has a fundamental Markoff triple  if and only if $i\in K$. Moreover, if $i\in K$ then a fundamental Markoff triple is of the form
 $$
 (0,\pm if,f)
 $$
 or 
 $$
 \left(\frac{2a}{A}, a f\pm \frac{2ai}{A},f\right)
 $$
 for some $f \in K[t]\backslash K$ and $a=\pm 1$.
\end{lemma}
\begin{proof}
It is easy to check that if $i\in K$ then $(0,it,t)$ is an example of a fundamental Markoff triple.

For the converse, we may assume from Lemma \ref{lem:xzero} that $x\in K^*$, for any fundamental triple $(x,y,z)$. Consequently, $y$ and $z$ are non-zero, non-constant polynomials. Since $\deg y=\deg z$, we can use long division to find $a\in K^*$ and $b\in K[t]$ with $\deg b<\deg z$ such that $y=az+b$. From \eqref{eq:Markoffpol}, we arrive at
\begin{equation}\label{eq:yeqz}
(a^2+1)z^2+2abz+b^2+x^2=Aaxz^2+Abxz.
\end{equation}
Notice that $b=0$ would imply that $(a^2+1-Aax)z^2=-x^2$, which contradicts the fact that $z$ is non-constant. Thus $b\neq 0$.

Let $a_z\neq 0$ be the leading coefficient of $z$. Since $\deg b<\deg z$, we see that the leading coefficients of the left- and right-hand side of \eqref{eq:yeqz} are $(a^2+1)a_z^2$ and $Aaxa_z^2$, respectively. This implies that
\begin{equation}\label{eq:aval}
a^2+1=Aax 
\end{equation}
and, from \eqref{eq:yeqz},
$$
2abz+b^2+x^2=Abxz.
$$
Similarly, given that $\deg b<\deg z$, a comparison of leading coefficients in the previous equality lead us to
\begin{equation}\label{eq:bval}
 2ab=Abx
\end{equation}
and 
\begin{equation}
 b^2+x^2=0.
\end{equation}
The last equality implies that $i\in K$,  $b\in K^*$ and $b=\pm i x$. Also, \eqref{eq:aval} and \eqref{eq:bval} imply that $a^2=1$ and $x=2a/A$. In conclusion, if we let $f=z$ then a fundamental triple with $x\neq 0$ is of the form
$$
 \left(\frac{2a}{A}, af\pm \frac{2ai}{A},f\right).
$$
\end{proof}
 
This result finishes the proof of Theorem \ref{thm:pol_tree}. Indeed, when combined with Lemma \ref{lem:orbit} it proves part (3) and part (1) of Theorem \ref{thm:pol_tree}, for $A$ constant.

\section{Counting the number of solutions of bounded height}\label{sec:countingsol}

In this section, we develop a machinery that can be used to compute the size and asymptotics of the following sets
$$
\{T=(\deg x,\deg y,\deg z):(x,y,z)\in M_A(K[t]), \max T\leq H\},
$$
for general $K$, and
$$
\{(x,y,z)\in M_A(K[t]): \max (\deg x,\deg y,\deg z)= H\}
$$
when $K$ is finite. 

In our discussion we use notation from  Section \ref{sec:Markoff_tree} and  fix a non-zero polynomial $A$ of degree $\beta$. We also define the \emph{signature of a Markoff triple $P=(x,y,z)$} as the triple $S(P)=(\deg x,\deg y,\deg z)$. 
Given a positive integer $n$, our first step is to evaluate the function
\begin{equation}
\Ccal_A(n)=\#\{S(P):P\in M_A(K[t]),h(P)=n\}, 
\end{equation}
which counts  the number of signatures of Markoff triples of given height $n$. For that matter, we reinterpret Theorem \ref{thm:pol_tree} in graph-theoretical terms as below.
 
When composed with certain permutation of coordinates, the automorphism $\rho$ defined by \eqref{eq:pred} yields the following ``branching operations''
\begin{equation} \label{eq:branching}
\begin{minipage}{0.5\textwidth}
\centering
 \begin{forest}
  for tree={scale=1,grow=east,
growth parent anchor=east,
parent anchor=east,
child anchor=west}
        	  [{${P=(x,y,z)}$}
		[${\sigma_2(P)=(x,Axz-y,z)}$
	 ]
        		[${\sigma_1(P)=(Azy-x,y,z)}$ 
		]
      ] 
  \end{forest}
  \end{minipage}
\end{equation}
One can use Lemmas \ref{lem:funsol1} and \ref{lem:funsol2} to observe that if $P$ is a fundamental triple then both $\rho(P)$ and $\sigma_2(P)$ are fundamental triples with the same height as $P$, while $\sigma_1(P)$ is not a fundamental triple. Moreover, if $P$ is  a non-fundamental Markoff triple then 
$$
\rho(P)<h(P)<h(\sigma_1(P)),h(\sigma_2(P)).
$$
Therefore, the branching operations in \eqref{eq:branching} can be used to construct an infinite binary tree of Markoff triples. For instance, if $A=1$  then $Q=(t,t+2i,t^2+2i t-2)$ and \eqref{eq:branching} generate the tree in Figure \ref{fig:Markofftreepol}.
\begin{figure}[ht]
\centering
\begin{adjustbox}{width=\linewidth}
\begin{forest}
for tree={anchor=base west,grow=east,
growth parent anchor=east,
parent anchor=east,
child anchor=west,}
[${(t,t+2i,t^2+2i t-2)}$ 
              [${(t+2i,t^2+2i t-2,\text{$t^{3} + 4i t^{2} -7 t - 4i$})}$, tier=l1 
                      [$({t+2i,t^{3} + 4i t^{2} -7 t - 4i,t^4 + 6i t^3 - 16t^2 - 20i t + 10)}$, tier=l2   [$\cdots$, tier=l3 ][$\cdots$, tier=l3 ] ]
           [${(t^2+2i t-2,t^{3} + 4i t^{2} -7 t - 4i,t^{5} + 6i t^{4} - 17 t^{3} - 26i t^{2} + 21t + 6i)}$, tier=l2  [$\cdots$, tier=l3 ][$\cdots$, tier=l3 ] ] ]
          [${(t,t^2+2i t-2,t^3+2i t^2-3t-2i)}$, tier=l1 
            [${(t^2+2i t-2,t^3+2i t^2-3t-2i,t^{5} + 4i t^{4} -9t^{3} - 12i t^{2} +9 t + 4i)}$, tier=l2   [$\cdots$, tier=l3 ][$\cdots$, tier=l3 ]
           ]
            [${(t,t^3+2i t^2-3t-2i,t^4+2i t^3-4t^2-4i t+2)}$, tier=l2  [$\cdots$, tier=l3  ][$\cdots$, tier=l3 ]]
             ]
      ] 
\end{forest}
\end{adjustbox}
\caption{Tree with root $(t,t+2i,t^2+2i t-2)$} \label{fig:Markofftreepol}
\end{figure}
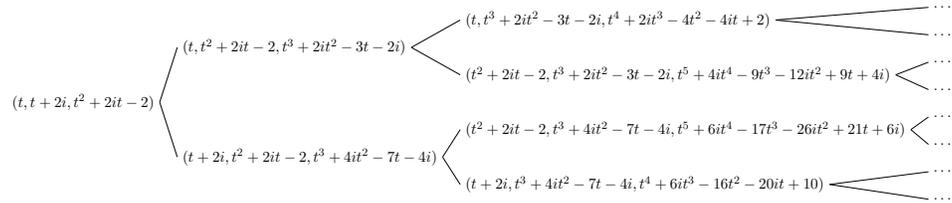 
Observe that the predecessor  of $Q$ is the fundamental triple $(2,t+2i,t)$. Consequently,  $Q=\sigma_1(2,t+2i,t)$ is the non-fundamental triple of smallest height on the  tree in Figure \ref{fig:Markofftreepol}.

Under this interpretation,  we can paraphrase Theorem \ref{thm:pol_tree} as saying that any Markoff triple  lies on a tree generated by the automorphisms \eqref{eq:branching} and rooted at
\begin{equation}\label{eq:fundsol}
 (0,a if,f)  \text{ or } \left(\frac{2a}{A},  af\pm \frac{2ai}{A},f\right)
\end{equation}
for some $f \in K[t]\backslash K$ and $a=\pm 1$. As done in Figure \ref{fig:Markofftreepol}, it is more convenient to assume that any tree of Markoff triples is rooted at the non-fundamental Markoff triple of smallest height. These triples are obtained from an application of $\sigma_1$ to \eqref{eq:fundsol}  and they are of the form
\begin{equation}\label{eq:nonfundsol}
 ( f,iaf,iaAf^2) \text{ or }  \left( f,a f\pm \frac{2ai}{A}, Aaf^2\pm {2aif}-\frac{2a}{A}\right)
\end{equation}
for some $f \in K[t]\backslash K$ and $a=\pm 1$.
 
The above discussion shows that the signature of a non-fundamental Markoff triple also lies on a tree. For instance, the tree of signatures of the Markoff triples in Figure \ref{fig:Markofftreepol} is given by Figure \ref{fig:10euclid} below. We call it the \emph{$(1,0)$-Euclid tree}, and it will play a crucial role in our remaining discussion.
\begin{figure}[H]
\centering
\begin{forest}
for tree={anchor=base west,grow=east,
growth parent anchor=east,
parent anchor=east,
child anchor=west,}
    [${(1,1,2)}$, [${(1,2,3)}$ 
              [${(1,3,4)}$, tier=l1 
                      [$\cdots$, tier=l2   ]
           [$\cdots$, tier=l2  ] ]
          [${(2,3,5)}$, tier=l1 
            [$\cdots$, tier=l2  
           ]
            [$\cdots$, tier=l2 ]
             ]
      ] ]
\end{forest}
\caption{$(1,0)$-Euclid Tree} \label{fig:10euclid}
\end{figure}
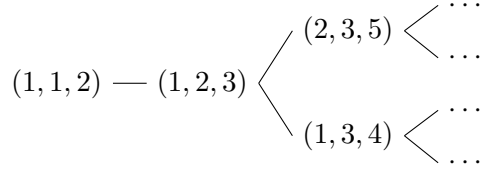 

More generally, let  $\alpha> 0$ be an integer and  let $(\alpha,\alpha,2\alpha+\beta)$ represent the signature of a non-fundamental triple in \eqref{eq:nonfundsol}. 
The \emph{$(\alpha,\beta)$-Euclid tree}  is the infinite binary tree (see Figure \ref{fig:euclidtree}) with root  $(\alpha,\alpha,2\alpha+\beta)$ and  branching operations\footnote{When $\beta=0$, trees of triple of integers given by these  branching operations have appeared in the literature under the name of Euclid trees because of their relationship with the euclidean algorithm, see for instance \cite{mcGinn}.} 
\begin{equation} \label{eq:euclid_branching}
\begin{minipage}{0.55\textwidth}
\centering
 \begin{forest}
  for tree={scale=1,grow=east,
growth parent anchor=east,
parent anchor=east,
child anchor=west}
        	  [{${T=(\tau_1,\tau_2,\tau_3)}$}
		[${\Gamma_{\beta,2}(T)=(\tau_1,\tau_3,\tau_1+\tau_3+\beta)}$
	 ]
        		[${\Gamma_{\beta,1}(T)=(\tau_2,\tau_3,\tau_2+\tau_3+\beta)}$
		]
      ] 
  \end{forest}
  \end{minipage}
\end{equation}
representing the signatures of the triples in \eqref{eq:branching}. 

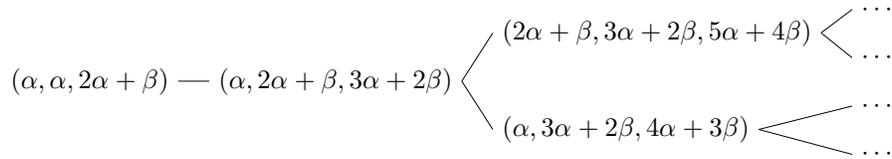
\begin{figure}[H]
\centering
\begin{adjustbox}{width=.95\linewidth}
\begin{forest}
for tree={anchor=base west,grow=east,
growth parent anchor=east,
parent anchor=east,
child anchor=west,}
    [${(\alpha,\alpha,2\alpha+\beta)}$, [${(\alpha,2\alpha+\beta,3\alpha+2\beta)}$ 
              [${(\alpha,3\alpha+2\beta,4\alpha+3\beta)}$, tier=l1 
                      [$\cdots$, tier=l2   ]
           [$\cdots$, tier=l2  ] ]
          [${(2\alpha+\beta,3\alpha+2\beta,5\alpha+4\beta)}$, tier=l1 
            [$\cdots$, tier=l2  
           ]
            [$\cdots$, tier=l2 ]
             ]
      ] ]
\end{forest}
\end{adjustbox}
\caption{The $(\alpha,\beta)$-Euclid Tree} \label{fig:euclidtree}
\end{figure} 
Our first result in this section relates $\Ccal_A(n)$ to
$$
\#\{T: \text{$T$ is on some $(\alpha,\beta)$-Euclid tree and $\max T=n$}\},
$$
which counts the number of triples on different $(\alpha,\beta)$-Euclid trees for which $n$ appear as a maximum. 
\begin{lemma}\label{lem:cacb}
 Let $n$ be a positive integer. If $\beta\neq 0$ then
$$
 \Ccal_A(n)=\#\{T: \text{$T$ is on some $(\alpha,\beta)$-Euclid tree and $\max T=n$}\}+1.
$$
 If $\beta=0$ then
$$
 \Ccal_A(n)=\#\{T: \text{$T$ is on some $(\alpha,\beta)$-Euclid tree and $\max T=n$}\}+2.
$$
\end{lemma}
\begin{proof}
 By construction, any triple on an $(\alpha,\beta)$-Euclid tree is the signature of some non-fundamental Markoff triple, and vice-versa. Therefore, the number of signatures of non-fundamental Markoff triples $P$ with height $h(P)=n\geq 1$ is given by
$$\#\{T: \text{$T$ is on some $(\alpha,\beta)$-Euclid tree and $\max T=n$}\}.$$ 
 
 To finish, we need to take into account the signatures of the fundamental triples with height $n$. According to  \eqref{eq:fundsol}, they are $(0,n,n)$ and $(-\infty,n,n)$.
 
Therefore, the result follows from Lemma \ref{lem:funsol1}  and Lemma \ref{lem:funsol2}.
\end{proof}

In light of the previous result, we define the function
$$
C_\beta(n)=\#\{T: \text{$T$ is on some $(\alpha,\beta)$-Euclid tree and $\max T=n$}\}+1.
$$
Notice that Lemma \ref{lem:cacb} implies
\begin{equation}\label{eq:cacb}
 \Ccal_A(n)=
\left\{\begin{array}{ll}
 C_\beta(n), & \text{if $\deg A\neq 0$}\\
  C_0(n) +1,& \text{if $\deg A= 0$}
\end{array}\right.
\end{equation}
In particular, $\Ccal_A(n)$ depends only on the degree of $A$, and not on $A$ itself.

Most of the remainder of this section  is used to compute $C_\beta(n)$. This will be done by describing the triples on a general $(\alpha,\beta)$-Euclid tree in terms of triples on the $(1,0)$-Euclid tree. The following definition will be useful in this task.

The \emph{$j$-th layer $L_j$} of the $(\alpha,\beta)$-Euclid tree is the subset of the vertices of the tree defined recursively by
$L_0=\{(\alpha,\alpha,2\alpha+\beta)\}$ and
$$
L_j=\{\Gamma_{\beta,1}(P),\Gamma_{\beta,2}(P): P\in L_{j-1}\},
$$
for $j\geq 1$. In other words, $L_j$ is the set of triples we obtain after iterating the branching operations $j$ times from $(\alpha,\alpha,2\alpha+\beta)$. Evidently, any triple on the $(\alpha,\beta)$-Euclid tree lies on some layer of the tree.

\begin{lemma}\label{lem:jlay}
 Let $\alpha>0$ and $j,\beta\geq 0$ be integers. 
 
$T$ is on the $j$-th layer of the $(\alpha,\beta)$-Euclid tree if and only if 
\begin{equation}\label{eq:jlay}
T=(b\alpha+(b-1)\beta,c\alpha+(c-1)\beta,d\alpha+(d-1)\beta), 
\end{equation}
for some triple $(b,c,d)$  on the $j$-th layer of the $(1,0)$-Euclid tree.  
\end{lemma}
\begin{proof}
We first prove by induction on $j$ that a triple on the $j$-th layer of the $(\alpha,\beta)$-Euclid tree is given by \eqref{eq:jlay}. Notice that the statement is true for $j=0$, since the $0$-th layers for both trees are $\{(\alpha,\alpha,2\alpha+\beta)\}$ and $\{(1,1,2)\}$. By definition and the induction hypothesis, any triple on the $(j+1)$-th layer of the $(\alpha,\beta)$-Euclid tree is of the form $\Gamma_{\beta,1}(T)$ or $\Gamma_{\beta,2}(T)$ where $T$ is given by  \eqref{eq:jlay} and $(b,c,d)$ is on the $j$-th layer of the $(1,0)$-Euclid tree. Since $\Gamma_{0,1}(b,c,d)=(c,d,c+d)$ and  $\Gamma_{0,2}(b,c,d)=(b,d,b+d)$  are on the $(j+1)$-th layer of the $(1,0)$-Euclid tree, 
$$
\Gamma_{\beta,1}(T)=(c\alpha+(c-1)\beta,d\alpha+(d-1)\beta, (c+d)\alpha+(c+d-1)\beta)
$$
and
$$
\Gamma_{\beta,2}(T)=(b\alpha+(b-1)\beta,d\alpha+(d-1)\beta, (b+d)\alpha+(b+d-1)\beta)
$$
have the desired form.

The converse is also proved by induction on $j$, with the base case $j=0$ being trivial. Assume $(b,c,d)$ is a  triple on the $(j+1)$-th layer of the $(1,0)$-Euclid tree. Then, by definition, there exists $(e,f,g)$ on the $j$-th layer of the $(1,0)$-Euclid tree such that either 
$$
(b,c,d)=(f,g,f+g)=\Gamma_{\beta,1}(e,f,g)
$$
or
$$
(b,c,d)=(e,g,e+g)=\Gamma_{\beta,2}(e,f,g).
$$ 
This shows that if $T$ is given by \eqref{eq:jlay} then $T=\Gamma_{\beta,1}(P)$ or $T=\Gamma_{\beta,2}(P)$, where
$$
(e\alpha+(e-1)\beta,f\alpha+(g-1)\beta,g\alpha+(g-1)\beta)
$$
 is on the $j$-th layer of the $(\alpha,\beta)$-Euclid tree, by the induction hypothesis. Consequently, $T$ is on the $(j+1)$-th layer of the $(\alpha,\beta)$-Euclid tree and the result follows.
\end{proof}
The next result uses Lemma \ref{lem:jlay} to reduce the computation of $C_\beta(n)$ to the following counting function on the $(1,0)$-Euclid tree
$$
E(n)=\#\{T: \text{$T$ is on the $(1,0)$-Euclid tree and $\max T=n$}\},
$$
for $n\geq 2$, and $E(1)=1$. 
\begin{lemma} \label{lem:cbeta}
Let   $\beta\geq 0$ and $n>0$ be  integers. Then
$$
C_\beta(n)=\sum_{\genfrac{}{}{0pt}{}{d \mid (n + \beta)}{\beta d <{n + \beta}}} E(d).
$$ 
\end{lemma}
\begin{proof}

To prove our result, we let $n$ be a positive integer. First notice that if $n=1$ then the conditions $\beta d<n+\beta$ and $d\mid (n+\beta)$ imply that $d=1$. Therefore, 
$$
C_\beta(1)=\sum_{\genfrac{}{}{0pt}{}{d \mid (1 + \beta)}{\beta d <{1 + \beta}}} E(d)=E(1)=1
$$ holds, and we may assume that $n\geq 2$.

Let $T$ be a triple on an $(\alpha,\beta)$-Euclid tree such that $\max T=n\geq 2$. Thus Lemma \ref{lem:jlay} shows that   there exists a unique non-negative integer $d$ such that  $n=d\alpha+(d-1)\beta$ and $d$ is the maximum of a triple on the $(1,0)$-Euclid tree. Moreover, by rewriting this expression for $n$ as $d(\alpha + \beta) = \beta + n$ we see that  $d \mid (n+\beta)$ and that $({n+\beta)}/{d}=\alpha + \beta > \beta$. Also notice that $d>1$, otherwise $n=\alpha$ is not a maximum on the $(\alpha,\beta)$-Euclid tree. In conclusion,
$$
\#\{T: \text{$T$ is on some $(\alpha,\beta)$-Euclid tree and $\max T=n$}\}\leq \sum_{\genfrac{}{}{0pt}{}{1\neq d \mid (n + \beta)}{\beta d <{n + \beta}}} E(d).
$$

On the other hand, suppose that $d>1$ is the maximum of a triple $(c,b,d)$ on the $(1,0)$-Euclid tree satisfying  $d \mid (n+\beta)$ and $\beta d<n+\beta$. Therefore, there exists a unique integer $r>0$ satisfying $dr=n+\beta$ and $\beta<r$. As a consequence of Lemma \ref{lem:jlay},  $n=d(r-\beta)+(d-1)\beta$ is the maximum of a triple on the $(r-\beta,\beta)$-Euclid tree and, consequently,
$$
\#\{T: \text{$T$ is on some $(\alpha,\beta)$-Euclid tree and $\max T=n$}\}= \sum_{\genfrac{}{}{0pt}{}{1\neq d \mid (n + \beta)}{\beta d <{n + \beta}}} E(d).
$$
Therefore,
$$
C_\beta(n)=\sum_{\genfrac{}{}{0pt}{}{1\neq d \mid (n + \beta)}{\beta d <{n + \beta}}} E(d)+1=\sum_{\genfrac{}{}{0pt}{}{1\neq d \mid (n + \beta)}{\beta d <{n + \beta}}} E(d)+E(1)=\sum_{\genfrac{}{}{0pt}{}{d \mid (n + \beta)}{\beta d <{n + \beta}}} E(d),
$$
as desired.
\end{proof}

In the next result we compute $C_0(n)$ explicitly. This in turn allows us to find an explicit formula for $E(n)$ and, consequently,  $C_\beta(n)$.
\begin{lemma}\label{lem:cofloor}
 Let $n$ be a positive integer. Then
 $$
 C_0(n)=\left\lfloor \frac{n}{2}\right\rfloor+1
 $$
\end{lemma}
\begin{proof}
We may assume that $n\geq 3$, as the desired formula is easily verified otherwise. Given \eqref{eq:cacb}, it is enough to show that
$$
\#\{T: \text{$T$ is on some $(\alpha,0)$-Euclid tree and $\max T=n$}\}=P_n,
$$
where $P_n$ is the  number of partitions of $n$ into two parts. Indeed, by definition, $P_n$ is the number of solutions of the equation $n=e+f$, where $e$ and $f$ are positive integers. Because of the linear and symmetric relationship between $e$ and  $f$, we conclude that  $P_n=\#\{e\in\Z:1\leq e\leq n/2\}=\lfloor n/2\rfloor$.

Let $T$ be  a triple on  some $(\alpha,0)$-Euclid tree for which $\max T=n$. Since $n>2$, $T$ is not the root of the $(\alpha,0)$-Euclid tree. In particular, $T=\Gamma_{0,1}(T')$ or $T=\Gamma_{0,2}(T')$ for some $T'=(b,c,d)$ on  the $(\alpha,0)$-Euclid tree. Consequently, \eqref{eq:euclid_branching} shows that either $n=b+d$ or $n=c+d$. In any case,  to the given  triple $T$ we can associate a unique partition of $n$ into two parts and so
\begin{equation}\label{eq:partition}
\#\{T: \text{$T$ is on some $(\alpha,0)$-Euclid tree and $\max T=n$}\}\leq P_n. 
\end{equation}

Conversely, consider a partition of $n=e+f$, for integers $0<e\leq f$. Observe that to prove that equality holds in \eqref{eq:partition}, it is enough to show that $T=(e,f,e+f)$ is on the $(\alpha,0)$-Euclid tree. For that matter, define the function
$$
\gamma(\tau_1,\tau_2,\tau_3)=\left\{
\begin{array}{ll}
 (\tau_1,\tau_2-\tau_1,\tau_2), & \text{if $\tau_2\geq \tau_1$}\\
  (\tau_2,\tau_1-\tau_2,\tau_1), & \text{if $\tau_2\leq \tau_1$}
\end{array}
\right.
$$  
Clearly,  for any $k\geq 0$, 
\begin{equation}\label{eq:gammak}
\gamma^k(T)=(\tau_1,\tau_2,\tau_1+\tau_2) 
\end{equation}
for some integers $\tau_1,\tau_2$.  This shows that either $
\Gamma_{0,1}(\gamma^{k+1}(T))=\gamma^k(T)
$ or 
$
\Gamma_{0,2}(\gamma^{k+1}(T))=\gamma^k(T)
$. Consequently, if we use the branching operations $\Gamma_{0,1}$ and $\Gamma_{0,2}$ to create a tree rooted at $\gamma^k(T)$, for some $k\geq0$, then $T$ is  on this tree. Therefore, $T$ is on some $(\alpha,0)$-Euclid tree, if there are  integers $k\geq 0$ and $\alpha>0$ such that $\gamma^k(T)=(\alpha,\alpha,2\alpha)$.

Note that if $P=(\tau_1,\tau_2,\tau_3)$ is such that $\tau_1\neq \tau_2$ then
$$
\max\gamma(P)<\max P.
$$
By the well-ordering principle, we cannot have $\max\gamma^{k+1}(T)<\max \gamma^k(T)$, for all $k\geq 0$. This fact and \eqref{eq:gammak} imply the existence of integers $k\geq 0$ and $\alpha>0$ such that
$$
\gamma^k(T)=(\alpha,\alpha,2\alpha),
$$
as desired.
\end{proof}
\begin{corollary}
  Let $n$ be a positive integer. Then:
 $$
E(n)=\sum_{d\mid n}\mu(d)\left(\left\lfloor \frac{n}{2d}\right\rfloor+1\right),
 $$
 where $\mu(n)$ is the M\"obius function.
\end{corollary}
\begin{proof}
 If we combine Lemmas \ref{lem:cbeta} and \ref{lem:cofloor}, then
 $$
 \left\lfloor \frac{n}{2}\right\rfloor+1=\sum_{d\mid n} E(d).
 $$
 The result follows from an application of the M\"obius inversion formula.
\end{proof}

As discussed in the introduction, when $\beta =0$, we obtain the following asymptotic formula that bears a striking resemblance with the results of Silverman and Zagier.

\begin{corollary}[Theorem \ref{thm:counting}]\label{cor:counting}
 If $A$ is a non-zero constant then
 $$
\#\{S(P):P\in M_A(k[t]),h(P)\leq H\} \sim \frac{1}{4}H^2,
 $$
 as $H\longrightarrow \infty$.
\end{corollary}
\begin{proof}
We may assume that $H$ is an integer. By definition, 
$$
\#\{S(P):P\in M_A(k[t]),h(P)\leq H\}=\sum_{n=1}^{H} \Ccal_A(n).
$$ 
Since $\deg A=\beta =0$,  \eqref{eq:cacb} and Lemma \ref{lem:cofloor} imply
$$
\sum_{n=1}^{H} \Ccal_A(n)= \sum_{n=1}^{H} \left(\left\lfloor\frac{n}{2}\right\rfloor+2\right).
$$
Since $x-1<\lfloor x\rfloor\leq x$, we arrive at
$$
\frac{1}{4}(H^2+5H)=\sum_{n=1}^{H} \left(\frac{n}{2}+1\right) < \sum_{n=1}^{H} \Ccal_A(n)\leq \sum_{n=1}^{H} \left(\frac{n}{2}+2\right)=\frac{1}{4}(H^2+9H),
$$
from which the result follows.
\end{proof}

\begin{corollary}
Suppose $K$ is a finite field with $q$ elements. If $A$ is non-constant  with $\deg A=\beta$ and $n$ is a positive integer then
$$
\#\{(x,y,z)\in M_A(K[t]):h(x,y,z)=n\}=4(q-1)\sum_{\genfrac{}{}{0pt}{}{d \mid (n + \beta)}{\beta d <{n + \beta}}} 
 q^{\frac{n + \beta}{d}-\beta} E (d)
$$
\end{corollary}
\begin{proof}
 We follow parts of the argument used  in Lemma \ref{lem:cbeta}. Since $A$ is non-constant, from Lemma \ref{lem:funsol1} we have two choices for the form of a fundamental triple
 $$
 R=(0,\pm if, f),
 $$
  with $f\in K[t]\backslash K$. Both forms have signature $(-\infty,\alpha,\alpha)$ and yield an $(\alpha,\beta)$-Euclid tree rooted at $(\alpha,\alpha,2\alpha+\beta)$. When one is interested in counting Markoff triples instead of their signatures, the equality
  $$
  \Gamma_{\beta,1}(\alpha,\alpha,2\alpha+\beta)=\Gamma_{\beta,2}(\alpha,\alpha,2\alpha+\beta)
  $$
needs to be taken into account.  This issue is evident when you compare Figure \ref{fig:Markofftreepol} with Figure \ref{fig:10euclid}. Except for the $0$-th layer, every layer in Figure \ref{fig:Markofftreepol}  has twice the number of triples than the respective layer in Figure \ref{fig:10euclid}. Therefore, for a fixed $f$, we need to associate four $(\alpha,\beta)$-Euclid trees to the Markoff tree rooted at $\sigma_1(R)$.
  
Consequently, if we fix $n>0$ and $\deg f=\alpha$, we have $4E(d)$ non-fundamental Markoff triples $P$ with $h(P)=n$, where $d>1$ is an integer satisfying
\begin{equation}\label{eq:ndal}
  n=d\alpha+(d-1)\beta,
\end{equation}
  and $\beta d<n+\beta$. Moreover, \eqref{eq:ndal} shows that $d\mid(n+\beta)$ and
  $$
 \alpha=\dfrac{n+\beta}{d}+\beta.
  $$
  Since there are $(q-1)q^{\alpha}$ polynomials $f$ with $\deg f=\alpha$, we have that the number of non-fundamental Markoff triples with height $n$ is
\begin{equation}\label{eq:qd}
 \sum_{\genfrac{}{}{0pt}{}{1\neq d \mid (n + \beta)}{\beta d <{n + \beta}}} 
 4(q-1)q^{\alpha} E (d)= 4(q-1)\sum_{\genfrac{}{}{0pt}{}{1\neq d \mid (n + \beta)}{\beta d <{n + \beta}}} 
 q^{\frac{n + \beta}{d}-\beta} E (d).
\end{equation}

Given the form of $R$, it follows that the number of fundamental Markoff triples of height $n$ is
\begin{equation}\label{eq:qdum}
4(q-1)q^n= 4(q-1)q^{\frac{n + \beta}{d}-\beta}E(d)
\end{equation}
if we take $d=1$. Combining \eqref{eq:qd} and \eqref{eq:qdum}, we arrive at
  $$
\#\{(x,y,z)\in M_A(K[t]):h(x,y,z)=n\}=4(q-1)\sum_{\genfrac{}{}{0pt}{}{d \mid (n + \beta)}{\beta d <{n + \beta}}} 
 q^{\frac{n + \beta}{d}-\beta} E (d),
$$
as desired.
\end{proof}

When $A$ is a non-zero constant and $K$ is finite, a  result similar to the previous one can be proved. We leave the details to the interested reader.

\section{Acknowledgments}
This work was supported, in part, by the Cross-Disciplinary Science Institute at Gettysburg College (X-SIG)

\bibliographystyle{alpha}

\end{document}